\documentclass{amsart}

\usepackage{amsmath, amssymb, amsbsy}
\usepackage{thmtools}
\usepackage{array}
\usepackage{graphpap, paralist,xcolor}
\usepackage[mathscr]{eucal}
\usepackage{graphicx}

\usepackage{setspace}
\setdisplayskipstretch{1}
\usepackage{geometry}
\usepackage{comment}
\usepackage{mathtools}
\usepackage{float}
\usepackage[colorlinks,backref=page,citecolor=blue]{hyperref}
\usepackage[capitalize]{cleveref}
\usepackage{crossreftools}
\DeclarePairedDelimiter{\ceil}{\lceil}{\rceil}
\newcounter{bullet}

\newtheorem{thm}{Theorem}[section]
\newtheorem{prop}[thm]{Proposition}
\newtheorem{cor}[thm]{Corollary}
\newtheorem{lem}[thm]{Lemma}

\newtheorem{conj}[thm]{Conjecture}

\theoremstyle{definition}
\newtheorem{mydef}[thm]{Definition}
\newtheorem{claim}[thm]{Claim}

\newtheorem{question}[thm]{Question}
\newtheorem{remark}[thm]{Remark}
\newtheorem{fact}[thm]{Fact}
\newtheorem{obs}[thm]{Observation}

\crefname{lem}{Lemma}{Lemmas}

\newcommand{\EE}{\mathbb{E}}

\newcommand{\cD}{\mathcal{D} }

\newcommand{\cF}{\mathcal{F} }
\newcommand{\cG}{\mathcal{G} }

\newcommand{\cI}{\mathcal{I} }

\newcommand{\cS}{\mathcal{S} }

\newcommand{\beq}[1]{\begin{equation}\label{#1}}
\newcommand{\enq}[0]{\end{equation}}

\newcommand{\eps}{\epsilon}

\newcommand{\gd}[0]{\delta }
\newcommand{\gD}[0]{\Delta }

\newcommand{\nin}[0]{\noindent}

\newcommand{\sub}[0]{\subseteq}

\newcommand{\bin}[0]{\mbox{\rm{Bin}}}

\newcommand{\pr}[0]{\mathbb{P}}

\newcommand{\aut}[0]{\mbox{aut}}

\renewcommand{\d}{\mbox{\rm{degen}}}

\renewcommand{\ceil}[1]{\left\lceil#1\right\rceil}

\newcommand{\pE}[0]{p_{\mathbb E}}

\newcommand{\perm}{{\rm perm}}

\newenvironment{subproof}[1][\proofname]{
  
  \begin{proof}[#1]
}{
  \end{proof}
}

\begin{document}

\nin {\small \today}

\bigskip

\title{A reformulation of the discrete Convexity Conjecture via $k$-thresholds}

\author[R. Ascoli]{Ruben Ascoli}
\address{School of Mathematics, Georgia Institute of Technology}
\email{rascoli3@gatech.edu}

\author[X. He]{Xiaoyu He}
\address{School of Mathematics, Georgia Institute of Technology}
\email{xhe399@gatech.edu}

\author[J. Park]{Jinyoung Park}
\address{Department of Mathematics, The Courant Institute School of Mathematics, Computing, and Data Science, New York University}
\email{jinyoungpark@nyu.edu}

\author[M. Talagrand]{Michel Talagrand}
\address{French National Center for Scientific Research (CNRS)}
\email{michel.talagrand@gmail.com}

\begin{abstract}
We introduce the notion of ``$k$-thresholds'' and show that Talagrand's discrete convexity conjecture is equivalent to the assertion that, for some universal integer $k \ge 2$, the $k$-threshold of every increasing family is at most a universal constant times its expectation threshold. We prove a reduction theorem that bounds the $k$-threshold of any increasing graph property in terms of ordinary thresholds of graphs in suitable decompositions of its members. As a consequence, we determine, up to a constant factor, the $k$-threshold of every fixed graph in terms of a natural $k$-density parameter. We also prove that $k=2$ suffices for several classical spanning graph containment properties. More generally, we establish the conjectured comparison between $k$-thresholds and expectation thresholds for broad classes of graph containment properties whose target graphs have low degeneracy.
\end{abstract}

\maketitle

\section{Introduction}\label{sec.intro}

The convexity conjecture of the fourth author asks, roughly, whether one can ``create" convexity in a constant number of steps regardless of the dimension of the ambient space. We refer the reader to \cite{talagrand1995all} for background on the conjecture, and to \cite{HST26} for a recent proof by Hua, Song, and Tudose.
In his original paper~\cite{talagrand1995all}, the fourth author also introduced a discrete version of the conjecture, which he later revisited in slightly different forms in \cite{talagrand2010many, Tal26}. This discrete conjecture, stated as \Cref{conj:DCC} below, remains open and is the main focus of the present paper.

We approach the discrete convexity conjecture by introducing the new notion of $k$-thresholds, which allows us to reformulate the conjecture in the language of threshold phenomena. We then develop this viewpoint for graph properties by proving a reduction theorem that bounds $k$-thresholds in terms of ordinary thresholds. Using this reduction, we verify the reformulated conjecture for several broad classes of graph containment properties.

To state the discrete convexity conjecture, we first introduce some basic definitions. For a finite set $X$ and $p \in [0,1]$, let $X_p$ denote the random subset of $X$ obtained by including each element independently with probability $p$, and let $\mu_p$ denote its distribution on $2^X$. Thus, for every $A \sub X$, $\mu_p(\{A\})=p^{|A|}(1-p)^{|X \setminus A|}$, and $\mu_p(\mathcal A)=\pr(X_p \in \mathcal A)=\sum_{A \in \mathcal A}\mu_p(\{A\})$ for every $\mathcal A \sub 2^X$.
With a given universe $X$, a \textit{family} or \textit{property} $\cF$ is a subset of $2^X$, and $\cF$ is \textit{nontrivial} if $\cF \neq \emptyset, 2^X$. Unless explicitly stated otherwise, all families considered below are assumed to be nontrivial.
Say a family $\cF \sub 2^X$ is \textit{increasing} if $B \supseteq A \in \cF$ implies $B \in \cF$ and \textit{decreasing} if $B \sub A \in \cF$ implies $B \in \cF$. In this paper, $\cI$ (respectively $\cD$) will always denote an increasing (respectively decreasing) family.

\begin{mydef}
For a decreasing family $\cD \sub 2^X$ and positive integer $k$, 
\[\cD_{(k)}:=\{D_1 \cup \cdots \cup D_k: D_i \in \cD\}, \quad \text{and} \quad \cD^{(k)}:=2^X \setminus \cD_{(k)}.\]
In words, $\cD^{(k)}$ is the collection of sets that cannot be expressed as a union of $k$ sets in $\cD$. Note that $\cD_{(k)}$ is decreasing and $\cD^{(k)}$ is increasing.
\end{mydef}

For $S \sub X$, define the \textit{up-set}
generated by $S$ to be $\langle S \rangle:=\{T \sub X:T \supseteq S\}$. Given $\cI \sub 2^X$, we say that a family $\cG \sub 2^X$ \textit{covers} $\cI$ if $\cI \sub \bigcup_{S \in \cG} \langle S \rangle$.  Following \cite{talagrand1995all, talagrand2005generic, talagrand2010many}, we say that $\cI \sub 2^X$ is \textit{$p$-small} if it has a cover $\cG \sub 2^X$ such that $\sum_{S \in \cG} \mu_p(\langle S \rangle) \le 1/2$. Note that if $\cI$ is $p$-small, then
\beq{eq:p-small}
    \mu_p(\cI)
    \leq \mu_p\left(\bigcup_{S\in\cG}\langle S\rangle\right)
    \leq \sum_{S\in\cG}\mu_p(\langle S\rangle)
    \leq \frac12 .
\enq
Thus, $p$-smallness is stronger than the condition $\mu_p(\cI)\le 1/2$; it requires this small measure to be certified by a cover.

\begin{conj}[Talagrand \cite{talagrand1995all}]\label{conj:DCC}
    There exist a positive integer $k$ and a universal constant $L\ge 1$ such that, for every $p \in (0,1)$, finite set $X$, and decreasing family $\cD \sub 2^X$, if $\mu_p(\cD) \ge 1-1/(2k)$ then $\cD^{(k)}$ is $(p/L)$-small.\footnote{The original formulation of the discrete convexity conjecture asks
    for the stronger-looking conclusion that $\cD^{(k)}$ is $p$-small, with
    the assumption $\mu_p(\cD)\ge 1-1/k$. As noted in
    \cite{talagrand2010many}, allowing a constant-factor change in the
    parameter of smallness gives an equivalent conjecture, after changing the
    universal constant $k$. Similarly, replacing $1/k$ by $1/(2k)$ in the hypothesis also gives an equivalent formulation.}
\end{conj}

Observe that, under the assumptions of \Cref{conj:DCC}, the family $\cD^{(k)}$ already has small $\mu_p$-measure since $\mu_p(\cD^{(k)})\le 1 - \mu_p(\cD) \leq 1/(2k)$, and the content of the conjecture is the stronger conclusion that $\cD^{(k)}$ is $(p/L)$-small. In other words, apart from a family whose smallness is certified by a cover, every subset of $X$ can be written as a union of $k$ members of $\cD$.

We introduce the notion of $k$-thresholds and use it to give an equivalent formulation of \cref{conj:DCC}.

\begin{mydef}[$k$-threshold]\label{def:k-threshold} For an increasing family $\cI \sub 2^X$ and a positive integer $k$, the \textit{$k$-threshold} of $\cI$, denoted by $p_k(\cI)$, is the infimum of the set of $p\in[0,1]$ such that
\beq{eq:k-thr}\text{for any decreasing $\cD \sub 2^X$ with $\mu_p(\cD)\ge 1-1/(2k)$, we have $\cD_{(k)} \cap \cI \ne \emptyset$.}\enq 
\end{mydef}

\noindent Note that \eqref{eq:k-thr} holds for every $p >p_k(\cI)$. 
Indeed, as $p$ increases, $\mu_p(\cD)$ decreases for every decreasing family $\cD$, so the collection of families satisfying $\mu_p(\cD) \ge 1-1/(2k)$ can only shrink. Furthermore, observe that for any increasing family $\cI$, $p_k(\cI)$ is non-increasing in $k$ since $\cD_{(k)}\subseteq \cD_{(k+1)}$ for any decreasing family $\cD$ and positive integer $k$.

Recall that the \textit{threshold}\footnote{This definition is finer than the original notion
of Erd\H{o}s and R\'enyi \cite{erdos1960evolution}, 
which defines $p^*=p^*(n)$ to be 
\emph{a threshold function} for the property $\cI=\cI_n$ if $\mu_p(\cI) \rightarrow 0$ when $p \ll p^*$ and $\mu_p(\cI) \rightarrow 1$ when $p \gg p^*$.
That $p_c(\cI)$ is always a threshold function for $\cI$ follows from \cite{bollobas1987threshold}.} for $\cI$, $p_c(\cI)$, is the unique $p$ such that $\mu_p(\cI)=1/2$. (Existence and uniqueness of $p_c$ follow from the fact that, for every nontrivial increasing family $\cI$, $\mu_p(\cI)$ is continuous and strictly increasing in $p$.) When $k=1$ the condition \eqref{eq:k-thr} fails exactly when $\mu_p(\cI^c)\ge 1/2$, so $p_1(\cI)=p_c(\cI)$, and $k$-thresholds are a generalization of the usual notion of thresholds. Following \cite{kahn2007thresholds}, define the \emph{expectation threshold} of a property $\cI\subseteq 2^X$ by \begin{equation*}
    q(\cI) = \max\{p: \mathcal I\text{ is $p$-small}\}.
\end{equation*}
By \eqref{eq:p-small}, we have $q(\cI) \leq p_c(\cI)$ for any increasing $\cI$.

We now present the reformulation of \cref{conj:DCC} using $k$-thresholds. 
\begin{conj}
\label{conj:DCC_k}There exist an integer $k\ge 2$ and a universal constant $L \ge 1$ such that, for every finite set $X$ and increasing family $\cI \sub 2^X$,
\beq{eq:reform_conc} p_k(\cI) \le Lq(\cI).\enq
\end{conj}
\begin{prop}\label{prop:equiv}    \Cref{conj:DCC} and \Cref{conj:DCC_k} are equivalent.
\end{prop}
The proof is a short argument using the definitions, given in \Cref{sec:main proofs}.

\begin{remark}\label{rmk:DCC_tightness}
The bound in \cref{conj:DCC_k}, if true, is best possible up
to a constant factor. Indeed, let
$\cI=\{A\subseteq X:A\ne \emptyset\}$, where $|X|=n$.
Then $q(\cI)=1/(2n)$, since any $p$-small cover of $\cI$ must cover each
singleton. On the other hand, \eqref{eq:k-thr} fails exactly when
$\cD=\{\emptyset\}$ has $\mu_p(\cD)\ge 1-1/(2k)$, and hence
\[
    p_k(\cI)=1-\left(1-1/(2k)\right)^{1/n}=\Theta_k(1/n)=\Theta_k(q(\cI)).
\]
Thus, for any fixed $k$, the right-hand side of \eqref{eq:reform_conc} cannot be replaced by $o(q(\cI))$. 
\end{remark}

For comparison with known results, the recent resolution of the convexity conjecture by Hua, Song, and Tudose~\cite{HST26} implies that there exist an integer $k \ge 2$ and a universal constant $\alpha \in (0,1)$ such that $p_k(\cI) \le q(\cI)^\alpha$ for every finite set $X$ and increasing family $\cI \sub 2^X$. At $k=1$, 
the Kahn--Kalai Conjecture \cite{kahn2007thresholds}, proved by Pham and the third author \cite{park2024proof}, gives that $p_1(\cI) = O(q(\cI) \log \ell(\mathcal I))$ where $\ell(\cI)$ is the maximum of $2$ and the size of a largest minimal element of $\cI$. The logarithmic factor is necessary for certain families; 
see, e.g., \cite{frankston2021thresholds,park2023threshold,park2024proof}.  Consequently, \Cref{conj:DCC_k} cannot hold with $k=1$.

\subsection{Reduction to ordinary thresholds}\label{sec:main results1} Building on the reformulation, we prove a general reduction that bounds the $k$-threshold of any increasing \textit{graph property} in terms of ordinary thresholds of graphs in suitable decompositions of its members. We first present the relevant definitions.
We identify each element of $2^{\binom{[n]}2}$ with a labeled graph on the fixed vertex set $[n]$. We also consider graphs on at most $n$ vertices, not necessarily with vertex set $[n]$, typically as target graphs for containment properties or as subgraphs of graphs on $[n]$. For a graph $H$, let $v_H$ and $e_H$ denote its numbers of vertices and edges, respectively.  A graph  $J$ is a \textit{subgraph} of $H$ if it can be obtained from $H$ by deleting vertices, edges, or both. 
We say that $G$ \textit{contains a copy} of $H$ if $G$ has a subgraph isomorphic to $H$.

For two graphs $G, G' \in 2^{\binom{[n]}{2}}$, we say that $G'$ is a \emph{relabeling} of $G$ if there is a permutation $\pi$ of $[n]$ such that for $u,v\in [n]$, we have $\{u,v\}\in E(G)$ if and only if $\{\pi(u),\pi(v)\}\in E(G')$. 
A family of graphs $\cG\subseteq 2^{\binom{[n]}{2}}$ is a \emph{graph property} if $G\in \cG$ implies $G'\in\cG$ for all relabelings $G'$ of $G$. 
In particular, for a given graph $H$ on at most $n$ vertices, the \textit{graph containment property} $\cI_H~(\sub 2^{\binom{[n]}2})$ is the family of all graphs in $2^{\binom{[n]}{2}}$ containing a copy of $H$. 

We use the shorthands $q(H):=q(\cI_H), p_c(H):=p_c(\cI_H)$, and so forth; if $H$ is edgeless, then $\cI_H = 2^{\binom{[n]}2}$. Although we generally consider thresholds only for nontrivial families, it is convenient to adopt the convention that $p_k(H)=p_c(H) =q(H) = 0$ in this case.

It is sometimes convenient to enlarge a target graph $H$ to an $n$-vertex graph by adding $n-v_H$ isolated vertices. For a graph $H$ on at most $n$ vertices, let $H^+$ denote the resulting $n$-vertex graph. Since an $n$-vertex graph contains a copy of $H$ if and only if it contains a copy of $H^+$, we have 
$\cI_{H} = \cI_{H^+}$, so $p_c(H)=p_c(H^+)$ for every graph $H$.

A \emph{$k$-decomposition} of a graph $H$ is a sequence of subgraphs $(H_1, \ldots, H_k)$ of $H$ such that 
$E(H_1) \cup \ldots \cup E(H_k) = E(H)$. 
We formulate the definition in terms of edge sets so that isolated vertices of $H$ need not appear in any of the $H_i$. Accordingly, when discussing decompositions, we often write $H=H_1 \cup \cdots \cup H_m$ as shorthand for $E(H)=E(H_1) \cup \cdots \cup E(H_m)$, with isolated vertices understood to be ignored.

\begin{mydef}
Given a graph $H$, define $r_k(H)$ to be the minimum of $\max_{i \in [k]} p_c(H_i)$ over all $k$-decompositions $(H_1, \ldots, H_k)$ of $H$.
\end{mydef}

\noindent Note that, since adding or deleting isolated vertices does not affect ordinary thresholds or $k$-decompositions, $r_k(H)=r_k(H^+)$ for every $H$.

We now state the reduction theorem that underlies our applications. Although it applies to any increasing graph property $\cI$, in all applications in \Cref{sec:main results2}, we take $\cI$ to be a graph containment property $\cI_H$.

\begin{thm}\label{thm:k-thr}
Let $k$ be a positive integer and $\cI$ be an increasing graph property. Then, for any $H \in \cI$,
\[ p_{k}(\cI) \le r_k(H).\]
Consequently, $p_k(\cI)\le \min_{H\in\cI} r_k(H)$.
\end{thm}

\Cref{thm:k-thr} reduces the somewhat abstract problem of bounding $p_k(\cI)$ to a concrete graph-decomposition problem involving ordinary thresholds. In particular, a suitable decomposition of a single member $G \in \cI$ is enough to bound the $k$-threshold of the entire property $\cI$.

For graph containment properties, the reduction takes the following concrete form. Let $(H_1, \ldots, H_k)$ be a $k$-decomposition of a target graph $H$, and suppose that each $H_i$ appears in $G_{n,p}$ with probability greater than $1/2$. Equivalently, $p>p_c(H_i)$ for every $i \in [k]$, and hence $r_k(H) \le \max_{i \in [k]} p_c(H_i)<p$. 
Now applying \Cref{thm:k-thr} to $H^+ \in \cI_H$, we obtain
\[p_k(H) \le r_k(H^+)=r_k(H)<p.\] Equivalently, every decreasing family $\cD \sub 2^{\binom{[n]}2}$ with $\mu_p(\cD) \ge 1-1/(2k)$ has the property that $\cD_{(k)}$ contains a copy of $H$.

\subsection{Applications to graph containment properties}\label{sec:main results2} We apply \Cref{thm:k-thr} to obtain bounds for several classes of graph containment properties. Throughout this subsection, $n$ tends to infinity, and the target graph $H$ may depend on $n$ unless stated otherwise.

\subsubsection*{Fixed graphs}
We begin with the case in which 
$H$ is a \emph{fixed} graph, meaning that $H$ does not depend on $n$. 
In this setting, for every fixed positive integer $k$, we determine the correct order of magnitude of $p_k(H)$ by proving matching upper and lower bounds in terms of a natural $k$-density parameter $d_k(H)$.

Define 
the \textit{maximum density} of $H$ by
\[d(H)=\max\{e_J/v_J: J \text{ is a subgraph of $H$ with } v_J \ge 1\}.\]
Classical results of 
Erd\H{o}s and R\'enyi \cite{erdos1960evolution} and Bollob\'as \cite{bollobas1981random} show that the threshold for $H$-containment is determined by $d(H)$; namely, $p_c(H)=\Theta(n^{-1/d(H)})$ where the implicit constants depend only on $H$.\footnote{While this order-of-magnitude form is sufficient for our purpose, we note that much sharper results are known for fixed graphs; see, e.g., \cite[Chapter~3]{janson2011random} and the references therein.} We prove a $k$-threshold analogue of this theorem. 
\begin{mydef}
    Given a graph $H$, define the \textit{$k$-density} of $H$, $d_k(H)$, to be the minimum of $\max_{i \in [k]} d(H_i)$ over all $k$-decompositions $(H_1, \ldots, H_k)$ of $H$.
\end{mydef} 

\begin{thm}\label{thm:small graphs}  Let $H$ be a fixed graph with at least one edge. For every positive integer $k$,
\[p_k(H)=\Theta(n^{-1/d_k(H)}),\] 
where the implicit constants depend only on $H$ and $k$.
\end{thm}

\medskip

\subsubsection*{Beyond fixed graphs}

We next establish 
\Cref{conj:DCC_k} for several graph containment properties whose target graphs $H$ grow with $n$. 

We first prove \cref{conj:DCC_k} with $k=2$ when the target graph belongs to one of several classical spanning graph classes. 
Recall that an \textit{$F$-factor} on $[n]$ is a collection of $n/v_F$ copies of $F$ whose vertex sets partition $[n]$, where we always assume $v_F |n$. For example, if $F$ is a single edge, then an $F$-factor is a perfect matching. 

\begin{thm}\label{thm:spanning graphs}
Let $H$ be a spanning tree whose maximum degree is bounded by a fixed constant, 
the $m$-th power of a Hamiltonian cycle for some fixed $m \ge 1$, or an $F$-factor for some fixed graph $F$ with at least one edge. Then $p_2(H) = O(q(H))$, where the implicit constant may depend on the maximum degree bound, $m$, or $F$, as appropriate, but not on $n$.
\end{thm}

Allowing larger values of $k$, we next establish 
\cref{conj:DCC_k} for a broader class of graph containment properties.
Recall that the \emph{degeneracy} of a graph $H$, denoted by $\d(H)$, 
is the smallest integer $d$ such that every nonempty subgraph of $H$ has minimum degree at most $d$. Equivalently, 
$\d(H)$ is the smallest $d$ for which the vertices of $H$ admit an ordering such that every vertex has at most $d$ neighbors preceding it.

\begin{thm}\label{thm:bounded degeneracy}
For every positive integer  $D$, there exist a positive integer $k=k(D)$ and a constant $L=L(D)$ such that every graph  $H$ with $\d(H)\leq D$ satisfies $p_k(H) \le L q(H)$.
\end{thm}

We further extend this result to graphs with degeneracy up to order $\log n/\log\log n$, at the cost of imposing a maximum-degree condition.

\begin{thm}\label{thm:general graph containment}
For every $C, \eps>0$, there exist a positive integer $k=k(C,\eps)$ and a constant $L=L(C,\eps)$ such that every graph $H$ with $\d(H) \leq C\frac{\log n}{\log\log n}$ and maximum degree at most $n^{1-\epsilon}$ satisfies $p_k(H) \le Lq(H)$.
\end{thm}

\paragraph{\textbf{Paper organization and conventions}} We first prove \Cref{prop:equiv}, which establishes the equivalence of Conjectures \ref{conj:DCC} and \ref{conj:DCC_k}, together with the 
reduction theorem, \cref{thm:k-thr}, in \cref{sec:main proofs}. Then, we prepare for the proofs of Theorems \ref{thm:small graphs}--\ref{thm:general graph containment} with some graph-theoretic preliminaries in \cref{sec:preliminaries}. We prove Theorems \ref{thm:small graphs}, \ref{thm:spanning graphs}, \ref{thm:bounded degeneracy}, and \ref{thm:general graph containment} in Sections \ref{sec:small graphs}, \ref{sec:spanning graphs}, \ref{sec: bounded degeneracy}, and \ref{sec:general graph containment} respectively. We conclude with some discussion and open problems in \cref{sec:discussion}.

We always assume that $n$ is sufficiently large to support our asymptotic estimation. We do not make a serious attempt to optimize leading constants in our asymptotic results. As is standard, we omit floors and ceiling signs whenever they are not necessary for the proofs, and we assume that all large enough quantities are integers.

\section{Proofs of the reformulation and the reduction theorem}\label{sec:main proofs}

We first prove the equivalence of Conjectures \ref{conj:DCC} and \ref{conj:DCC_k}.

\begin{proof}[Proof of \Cref{prop:equiv}]
Suppose first that \Cref{conj:DCC} holds with constants $k$ and $L$. We show
that \Cref{conj:DCC_k} also holds with the same constants. Fix an increasing
family $\cI$ and take any $p<p_k(\cI)$. By the definition of $p_k$, there is a
decreasing family $\cD$ such that $\mu_p(\cD)\ge 1-\frac{1}{2k}$ and $\cD_{(k)}\cap \cI=\emptyset$.
Equivalently, $\cI\subseteq \cD^{(k)}$. By \Cref{conj:DCC}, $\cD^{(k)}$ is
$(p/L)$-small, and hence $\cI$ is also $(p/L)$-small. Thus $p/L\le q(\cI)$.
Since this holds for every $p<p_k(\cI)$, we have $p_k(\cI)\le Lq(\cI).$

Conversely, suppose \Cref{conj:DCC_k} holds with constants $k$ and $L$. We
prove \Cref{conj:DCC} with the same constants. Let $\cD$ be a decreasing
family with $\mu_p(\cD)\ge 1-\frac{1}{2k}$, and set $\cI:=\cD^{(k)}$. If $\cI=\emptyset$, there is nothing to prove. By
\Cref{conj:DCC_k}, we have $p_k(\cI)\le Lq(\cI)$, so to draw the conclusion that $\cI$ is $(p/L)$-small, it suffices to show that
$p\le p_k(\cI)$. If instead $p>p_k(\cI)$, then \eqref{eq:k-thr} holds for this $p$ (see the paragraph following \Cref{def:k-threshold}), so applying it to
$\cI$ and to the same family $\cD$ gives $\cD_{(k)}\cap \cD^{(k)}\ne\emptyset$,
which is impossible. Hence $p\le p_k(\cI)\le Lq(\cI)$, so $q(\cI)\ge p/L$.
This proves \Cref{conj:DCC}.
\end{proof}

In the proof of \cref{thm:k-thr} below, the fact that we consider graph properties will be crucial.\footnote{To be more precise, the facts that are crucially used for the proof are that the permutation group $\cS_n$ acts on the ground set $\binom{[n]}{2}$, that $\mu_p$ is invariant under this action, and that $\cI$ is invariant under relabeling.} 

\begin{proof}[Proof of \cref{thm:k-thr}]
Throughout the proof, we fix a positive integer $k$, and set
$X=\binom{[n]}{2}$. Let $\cI \sub 2^X$ be a given increasing graph property
and let $H \in \cI$. We will often use $\hat H$ to denote a relabeling of $H$ in
$2^X$.

Choose a $k$-decomposition $(H_1,\ldots,H_k)$ of $H$ attaining the minimum in the definition of $r_k(H)$. By adding isolated vertices where necessary, we regard each $H_i$ as a graph on the common vertex set $[n]$. This does not change $p_c(H_i)$, and we continue to denote the resulting graph by $H_i$. Set 
\[
    \beta=\max\{p_c(H_i):i\in[k]\}.
\]
It remains to prove that $p_k(\cI)\le \beta$.  
If $\beta=1$, this is immediate, 
so assume $\beta<1$. We will show that \eqref{eq:k-thr} holds for every
$p\in(\beta,1)$, which implies the desired inequality.

Fix such a $p$. The following lemma establishes \eqref{eq:k-thr} for the fixed $p$. Let $\cD$ be a decreasing family (note that we do not
require $\cD$ to be a graph property) such that $\mu_p(\cD)\ge 1-\frac{1}{2k}$.

\begin{lem}\label{lem:k-thr_key}
There exist $\hat H_1,\ldots,\hat H_k\in\cD$ such that
\[
\text{(i) $\hat H_i$ is a relabeling of $H_i$ for each $i\in[k]$; \quad and \quad
(ii) $\bigcup_{i \in [k]} E(\hat H_i)$ is the edge set of a relabeling of $H$.}
\]
\end{lem}

\noindent To see that \eqref{eq:k-thr} follows from  \Cref{lem:k-thr_key} for the fixed $p$, suppose $(\hat H_1, \ldots, \hat H_k)$ satisfy the conclusion of \Cref{lem:k-thr_key}, and let $G^*$ be the graph on $[n]$ with edge set $E(G^*)=\bigcup_{i \in [k]} E(\hat H_i)$. Since each $\hat H_i$ belongs to $\cD$, we have $G^* \in \cD_{(k)}$. By part~\textit{(ii)}, $G^*$ is a relabeling of $H$. Since $H \in \cI$ and $\cI$ is invariant under relabeling, $G^* \in \cI$. Therefore, $\cD_{(k)} \cap \cI \ne \emptyset$, as required in \eqref{eq:k-thr}.

\begin{mydef}
For any graph $F\in 2^X$ and a permutation $\pi\in\cS_n$, define the graph
$T_\pi(F)\in 2^X$ by
\[
    E(T_\pi(F))=\{\{\pi(x),\pi(y)\}:\{x,y\}\in E(F)\}.
\]

The key quantity in the proof is the following. For each $F\in 2^X$, define
\[
    \perm(F):=\frac{|\{\pi\in\cS_n:T_\pi(F)\in\cD\}|}{|\cS_n|}.
\]
Thus, $\perm(F)$ is the probability for the event  $\{T_{\boldsymbol{\pi}}(F) \in \cD\}$ where $\boldsymbol{\pi}$ is chosen uniformly from $\cS_n$. Notice that it does not count copies of $F$ in the host graph; rather, it measures the proportion of relabelings of $F$ that belong to the family $\cD$. By symmetry,
\beq{eq:symm}
\text{$\perm(F')=\perm(F'')$ for any relabelings $F',F''$ of $F$.}
\enq
\end{mydef}

\Cref{lem:k-thr_key} will easily follow from the next lemma.

\begin{lem}\label{lem:perm_key}
For every $i\in[k]$, $\perm(H_i)>1-1/k$.
\end{lem}

\begin{proof}[Proof of \Cref{lem:k-thr_key} assuming \Cref{lem:perm_key}]
For each $\pi\in\cS_n$, the graphs $\{T_\pi(H_i):i\in[k]\}$ form a
$k$-decomposition of $T_\pi(H)$, which is a relabeling of $H$. Therefore,
\Cref{lem:k-thr_key} will follow if we show that
\beq{eq:pi_exist}
\text{there exists a $\pi_0\in\cS_n$ such that
$\{T_{\pi_0}(H_i):i\in[k]\}\sub \cD$.}
\enq

Let $\boldsymbol{\pi}$ be a uniformly random element of $\cS_n$ and let $Q_i$
be the event that $T_{\boldsymbol{\pi}}(H_i)\in \cD$. Then $\pr(Q_i)=\perm(H_i)$, so by \Cref{lem:perm_key}, we have $\pr(Q_i^c)<1/k$ for every $i\in[k]$.
Therefore,
\[
\pr\bigl(\bigwedge_{i\in[k]} Q_i\bigr)
=1-\pr\bigl(\bigvee_{i\in[k]} Q_i^c\bigr)
\ge 1-\sum_{i\in[k]}\pr(Q_i^c)>0.
\]
This yields \eqref{eq:pi_exist}.
\end{proof}

\begin{proof}[Proof of \Cref{lem:perm_key}]
We begin with the following observation.

\begin{obs}\label{obs:perm_subgr}
Let $G$ and $F$ be two graphs in $2^X$. If $F$ is a subgraph of $G$, then
$\perm(F)\ge \perm(G)$.
\end{obs}

\begin{subproof}
Let $P(G)=\{\pi\in\cS_n:T_\pi(G)\in\cD\}$
and define $P(F)$ similarly. We claim that $P(G)\sub P(F)$. To see this, note that for any $\pi\in\cS_n$, if $F$ is a
subgraph of $G$, then $T_\pi(F)$ is a subgraph of $T_\pi(G)$. Since $\cD$ is
decreasing, if $T_\pi(G)\in\cD$, then $T_\pi(F)\in\cD$. This shows that
$P(G)\sub P(F)$, and the observation follows.
\end{subproof}

Next, we use an averaging argument. The integral below means expectation
with respect to $F\sim \mu_p$; explicitly, for any function $f:2^X\to\mathbb R$,
\[
    \int f(F)\,d\mu_p(F)
    =
    \sum_{F\in 2^X} f(F)\mu_p(\{F\}).
\]

\begin{claim}\label{cl:perm_int} For any $F \in 2^X$,
    \[\int\perm(F)d\mu_p(F)=\mu_p(\cD).\]
\end{claim}

\begin{subproof} Let $\mathbf 1_{\cD}$ denote the indicator function of $\cD$; that is, $\mathbf 1_{\cD}(G)=1$ if $G \in \cD$, and $\mathbf 1_{\cD}=0$ otherwise. Using the definition $\perm(F)=\frac{1}{n!}\sum_{\pi\in\cS_n}{\bf 1}_{\cD}(T_\pi(F))$, we have
\[
\begin{aligned}
    \int \perm(F)\,d\mu_p(F)
    &= \frac{1}{n!}\sum_{\pi\in\cS_n}
       \int {\bf 1}_{\cD}(T_\pi(F))\,d\mu_p(F).
\end{aligned}
\]
For each fixed $\pi\in\cS_n$, the map $F\mapsto T_\pi(F)$ is a bijection of
$2^X$. Moreover, $\mu_p\bigl({T_\pi(F)}\bigr)=\mu_p({F})$ for every $F \in 2^X$, since $T_\pi(F)$ has the same number of edges as $F$. 
Hence, by the change of variables $F'=T_\pi(F)$,
\[
\int {\bf 1}_{\cD}(T_\pi(F))\,d\mu_p(F)
=\int {\bf 1}_\cD(F')\,d\mu_p(F')=\sum_{F'\in 2^X}{\bf 1}_{\cD}(F')\mu_p(\{F'\})
=\mu_p(\cD).
\]
(Note that the above uses only the invariance of $\mu_p$ under relabeling, without requiring any symmetry of $\cD$.) Therefore,
\[\int \perm(F) d\mu_p(F)=\frac{1}{n!}\sum_{\pi \in \cS_n} \mu_p(\cD)=\mu_p(\cD).\qedhere\]
\end{subproof}

Let $\cG:=\{F\in 2^X:\perm(F)>1-1/k\}$.

\begin{claim}\label{cl:cG_half}
    $\mu_p(\cG)\ge \frac12$.
\end{claim}

\begin{subproof}
By the assumption on $\cD$ and \Cref{cl:perm_int},
\[
    \int \perm(F)\,d\mu_p(F)=\mu_p(\cD)\ge 1-\frac{1}{2k}.
\]

For the sake of contradiction, suppose $\mu_p(\cG)<1/2$. Since $0\le \perm(F)\le 1$ for every
$F \in 2^X$ and $\perm(F)\le 1-1/k$ for $F \in \cG^c$, we have
\[
\begin{aligned}
    \int \perm(F)\,d\mu_p(F)
    &= \int_{\cG}\perm(F)\,d\mu_p(F)
       +\int_{\cG^c}\perm(F)\,d\mu_p(F) \\
    &\le \mu_p(\cG)+(1-1/k)(1-\mu_p(\cG))  \\
    &= 1-\frac{1}{k}\cdot(1-\mu_p(\cG))
     < 1-\frac{1}{2k},
\end{aligned}
\]
a contradiction. Hence $\mu_p(\cG)\ge 1/2$. \end{subproof}

Now we derive the conclusion of \Cref{lem:perm_key} that
\[ \text{for every $i \in [k]$, $\perm(H_i) > 1-1/k$.}\]
Fix $i\in[k]$.  
Since $p>\beta\ge p_c(H_i)$, we have $\mu_p(\cI_{H_i})>1/2$ if $H_i$ has at least one edge, while $\mu_p(\cI_{H_i})=1$ if $H_i$ is edgeless. In either case, $\mu_p(\cI_{H_i})>1/2.$ 
 Together with  \Cref{cl:cG_half}, this implies that 
$\cG\cap \cI_{H_i}\ne\emptyset$. Choose $G_i \in \cG \cap \cI_{H_i}$. Since $G_i\in \cG$, we have
$\perm(G_i)>1-1/k$. Since $G_i\in \cI_{H_i}$, it contains a relabeling $H_i'$ of $H_i$. Now by \Cref{obs:perm_subgr}, $\perm(H_i')\ge \perm(G_i)>1-1/k$, and finally, \eqref{eq:symm} gives $ \perm(H_i)=\perm(H_i')>1-1/k.$ 
This proves \Cref{lem:perm_key}.
\end{proof}

Since $p>\beta$ was arbitrary, it follows that $p_k(\cI) \le \beta=r_k(H)$, as desired.
\end{proof}

\section{Graph-theoretic preliminaries}\label{sec:preliminaries}

In this section, we collect some preliminary results.

\subsection{Graphic expectation thresholds and related estimates}\label{subsec:exp_thr}

To prove the results in \Cref{sec:main results2}, we need a way to relate ordinary thresholds $p_c$ to expectation thresholds $q$. Since $q(H)$ is defined through an optimization over all covers of $\cI_H$, it is generally difficult to compute directly. We therefore recall the  graphic expectation threshold\footnote{ Strictly speaking, \cite{kahn2007thresholds} defines $\pE$ with 1 in place of 1/2 and calls it the ``expectation threshold.''} from \cite{kahn2007thresholds}, which provides a useful lower bound on $q(H)$.\footnote{Much of the literature uses the ``fractional expectation threshold'' $q_f$ instead of $q$. This more tractable parameter satisfies $q(H) \le q_f(H)$, but this inequality goes in the opposite direction from what we need, since our applications require a lower bound on $q(H)$.}

\begin{mydef}
For a graph $J$, let $X_J$ denote the number of copies of $J$ in $G_{n,p}$, and write $\EE_pX_J$ for its expectation. 
The \emph{graphic expectation threshold} of a graph $H$ is given by
\beq{def:pE} \pE(H) = \min\{p \in [0,1]: \EE_p X_I \geq 1/2\ \text{ for every subgraph $I$ of $H$}
\}.\enq
\end{mydef}
\nin Note that \[\pE(H) \leq q(H) \text{ 
for any graph $H$;}\] indeed, let $q:=\pE(H)$ and choose a subgraph $I$ of $H$ such that $\mathbb E_qX_I=1/2$. 
Then the family $\cG_I \sub 2^{\binom{[n]} 2}$ that consists of all subgraphs of $K_n$ that are isomorphic to $I$ covers $\cI_H$, and $\sum_{S \in \cG_I}\mu_p(\langle S \rangle)=\mathbb E_qX_I=1/2$. Therefore, $\cI_H$ is $q$-small.

The following theorem of Dubroff, Kahn, and the third author bounds $p_c(H)$ from above in terms of $\pE(H)$.

\begin{thm}[{\cite[Theorem 1.2]{dubroff2025second}}]\label{thm:DKP}
There is a universal constant $K$ such that for any graph $H$ with at most $n$ vertices, $p_c(H) \leq K\pE(H)\log^3n$.
\end{thm}

We note that directly from the definition of $\pE(H)$, for every non-edgeless $H$ we have
\beq{eq:pE-alternate} \pE(H)=\max\left\{\left(\frac{\aut(I)}{2(n)_{v_I}}\right)^{1/e_I}: \text{$I$ is a subgraph of $H$ with $e_I \ge 1$}\right\}, \enq
where $\aut(I)$ is the number of automorphisms of $I$ 
and $(n)_a = n(n-1)\cdots(n-a+1)$. Indeed, we have $\mathbb E_pX_I=(n)_{v_I}p^{e_I}/\aut(I)$, so the condition $\mathbb E_pX_I \ge 1/2$ is equivalent to $p \ge \left(\aut(I)/(2(n)_{v_I})\right)^{1/e_I}$.

For later estimates, it will be convenient to replace the falling factorial $n_{v_I}$ by the simpler quantity $n^{v_I}$. Define
\[\tilde \pE(H):=\max\left\{\left(\frac{\aut(I)}{n^{v_I}}\right)^{1/e_I}:\text{$I$ is a subgraph of $H$ with $e_I \ge 1$}\right\}.\] 
Deleting isolated vertices does not decrease the expression defining $\pE$ and $\tilde \pE$, so the maxima can be taken over subgraphs with no isolated vertices. For every such subgraph $I$, we have $v_I \le 2e_I$.
Since $(n/e)^{v_I}\le (n)_{v_I}\le n^{v_I}$ and $v_I/e_I \le 2$, we obtain
\[\frac12\tilde \pE(H) \le \pE(H) \le e^2\tilde \pE(H).\] Therefore, $\pE(H)$ and $\tilde \pE(H)$ differ by at most an absolute constant factor.

\subsection{Thresholds of bounded-degree trees}

For our bound on the $2$-threshold of bounded-degree spanning trees, we use the following result of Alon, Krivelevich, and Sudakov~\cite{AlKrSu}, which concerns thresholds for nearly-spanning trees. This result was later improved by Balogh, Csaba, Pei, and Samotij~\cite{BCPS2010}.

\begin{thm}[{\cite[Theorem 1.1]{AlKrSu}}]\label{thm:embedding_trees}
    For every $d \geq 2$ and $0 < \epsilon < 1/2$, there exists a constant $c = c(d,\epsilon)$ such that, with high probability,  $G_{n,c/n}$ contains a copy of every tree on at most $(1-\eps)n$ vertices with maximum degree at most $d$.
\end{thm}

\subsection{The Janson inequalities}
Our proof of the bound on the $2$-threshold of $F$-factors in \cref{thm:spanning graphs} uses the following forms of the Janson inequalities; see, for example, \cite[Theorems 8.1.1 and 8.1.2]{AlSp}.
\begin{thm}[Janson Inequality]\label{thm:Janson}
Let $\Omega$ be a finite ground set, and let $S$ be a random subset of $\Omega$ obtained by including 
each $x\in \Omega$ independently with probability $p_x$. Let $(A_i)_{i\in I}$ be subsets of $\Omega$, let $X_i$ be the indicator of the event $A_i\subseteq S$, and set $X := \sum_{i\in I} X_i$. For distinct $i, j\in I$, write $i\sim j$ if $A_i\cap A_j \neq \emptyset$. Define $\mu = \EE[X]$ and $\Delta = \sum_{(i,j):\,i\sim j} \EE[X_iX_j]$, where the sum is taken over ordered pairs. Then, $\pr(X=0) \leq e^{-\mu + \Delta/2}$.
\end{thm}
\begin{thm}[Extended Janson Inequality]\label{thm:extJanson}
Under the same assumptions  
as \cref{thm:Janson}, if $\Delta\geq \mu$, then we have $\pr(X=0) \leq e^{-\mu^2/(2\Delta)}$.
\end{thm}

\subsection{Low degeneracy and the Nash-Williams theorem} In the proof of \Cref{thm:bounded degeneracy}, we use the Nash-Williams theorem to decompose low-degeneracy graphs into forests. 
Recall that the \textit{arboricity} of a graph is the minimum number of forests into which its edges can be partitioned.

\begin{thm}[Nash-Williams \cite{nash1964decomposition}]\label{thm:NW} Let $H=(V,E)$ be a graph and $k \ge 1$ be an integer. The edge set $E$ can be partitioned into at most $k$ forests if and only if for every non-empty subset of vertices $U \sub V$, the induced subgraph $H[U]$ has at most $k(|U|-1)$ edges. 
\end{thm}

The following immediate consequence relates arboricity to degeneracy.

\begin{cor}\label{cor:NW}
    Let $H$ be a graph with the degeneracy at most $D$. Then the arboricity of $H$ is at most $D$.
\end{cor}

\begin{proof}
    Take any subgraph $J$ of $H$ with $v_J \ge 2$. Restrict a $D$-degeneracy ordering of $H$ to $J$. Then every vertex in $J$ has at most $D$ neighbors in $J$ preceding it in this ordering, and the first vertex has none, so $e_J \le D(v_J-1)$. This holds for every such $J$, so applying \Cref{thm:NW} yields the conclusion.
\end{proof}

We also need the following two elementary facts about degeneracy. 
Recall that $d(H)$ and $\d(H)$ denote, respectively, the maximum subgraph density and the degeneracy of $H$.
\begin{claim}\label{claim:degen vs mad}
For every graph $H$, we have $d(H) \leq \d(H)\leq 2d(H)$.
\end{claim}

\begin{proof}
    Use $D$ for $\d(H)$. For the first inequality, consider any subgraph $J$ of $H$ with $v_J \ge 1$. As in the proof of \Cref{cor:NW}, by considering a $D$-degeneracy ordering of $H$, we obtain that $e_J/v_J \le D$. Therefore, $d(H) \le D=\d(H)$.

    For the other inequality, choose a nonempty subgraph $J$ of $H$ whose minimum degree is precisely $D$. Since the average degree of $J$ is at least its minimum degree, $\d(H)=D \le 2e_J/v_J \le 2d(H).$
\end{proof}

\begin{claim} \label{claim: degen decomp}
For every positive integer $k$, every graph $H$ admits a $k$-decomposition $(H_1, \ldots, H_k)$ such that 
$\d(H_i) \leq \ceil{\d(H) / k}$ for every $i\in[k]$.
\end{claim}
\begin{proof}
Let $D:=\d(H)$, and fix an ordering $v_1, \ldots, v_{v_H}$ in which each vertex has at most $D$ neighbors preceding it. For each $j$, partition the preceding neighbors of $v_j$ into $k$ parts, each of size at most $\ceil{D/k}$, and assign the edges from $v_j$ to the vertices in the $i$-th part to $H_i$. Every edge is assigned exactly once when its later endpoint is considered, so $(H_1, \ldots, H_k)$ is a $k$-decomposition of $H$. In the same ordering, every vertex has at most $\ceil{D/k}$ preceding neighbors in each $H_i$. Hence $\d(H_i) \le \ceil{D/k}$ for every $i \in [k]$.
\end{proof}

\section{Proof of \texorpdfstring{\Cref{thm:small graphs}}{Theorem 1.10}}\label{sec:small graphs}

Throughout the proof, we fix $H$ and $k$ and suppress dependency of constants on them. Recall from the introduction that any fixed graph $H$ satisfies $p_c(H) = \Theta(n^{-1/d(H)})$, which gives $r_k(H)=\Theta(n^{-1/d_k(H)})$. The upper bound of \cref{thm:small graphs} then follows immediately from \Cref{thm:k-thr}. 

The lower bound follows directly from the proposition below:

\begin{prop}\label{prop:k-thr_lb}
    For any fixed graph $H$ with at least one edge and positive integer $k$, there is a constant $c$ with the following property. If $p \le cn^{-1/d_k(H)}$, then $p$ violates \eqref{eq:k-thr} for $\cI=\cI_H$.
\end{prop}

\begin{proof}
Let $N_H$ be the number of subgraphs of $H$ with at least one edge and no isolated vertices. Since every such subgraph is determined by its edge set, $N_H\le 2^{e_H}$.
Suppose $p \le cn^{-1/d_k(H)}$, where $c<1/(2kN_H)$. 
Deleting isolated vertices from a non-edgeless piece of a $k$-decomposition changes neither its edge set nor its maximum density. Therefore, we may assume throughout the proof that every non-edgeless piece of a $k$-decomposition has no isolated vertices.

For every $k$-decomposition $\mathfrak D=(H_1,\ldots,H_k)$ of $H$, choose a member $H(\mathfrak D)$ satisfying $d(H(\mathfrak D)) \ge d_k(H)$. Such a member exists by the definition of $d_k(H)$. Since $H$ has at least one edge, $d_k(H)>0$, so $H(\mathfrak D)$ has at least one edge.

Let $J$ be a subgraph of $H(\mathfrak D)$ attaining its maximum density; that is, $e_J/v_J = d(H(\mathfrak D))$. Letting $X_J$ count copies of $J$ in $G_{n,p}$, we have
\[\mathbb E_pX_J \le n^{v_J}p^{e_J} \le c^{e_J}n^{v_J-e_J/d_k(H)}\le c<\frac{1}{2kN_H}.\]
By Markov's inequality, $\mu_p(\cI_J)=\pr(X_J \ge 1)<1/(2kN_H)$. 
Since $J$ is a subgraph of $H(\mathfrak D)$, this implies
\[\mu_p(\cI_{H(\mathfrak D)})<\frac{1}{2kN_H}.\]
Set
\[\cI^*:=\bigcup\{\cI_{H(\mathfrak D)}:\text{$\mathfrak D$ is a $k$-decomposition of $H$}\},\qquad \cD:=2^{{\binom{[n]}2}}\setminus \cI^*.
\]
Then $\cD$ is decreasing. Also, there are at most $N_H$ possible graphs $H(\mathfrak D)$.  Hence $\mu_p(\cI^*)<1/(2k)$, and so $\mu_p(\cD)>1-1/(2k)$.

It remains to show that $\cD_{(k)} \cap \mathcal I_H=\emptyset$. 
For the sake of contradiction, suppose $D_1,\ldots,D_k\in\cD$ and the graph on $[n]$ with edge set $D_1 \cup \cdots \cup D_k$ contains a copy $H'$ of $H$. For each $i \in [k]$, let $H_i'$ be the subgraph of $H'$ with edge set $E(H_i')=D_i \cap E(H')$ and isolated vertices omitted. Then $(H_1',\ldots, H_k')$ is a $k$-decomposition of $H'$. 
Pulling this decomposition back along an isomorphism from $H$ to $H'$ 
gives a $k$-decomposition $\mathfrak D=(H_1, \ldots, H_k)$ of $H$. Then $H(\mathfrak D)$ is one of the $H_i$, and the corresponding $D_i$ contains a copy of $H(\mathfrak D)$. Thus $D_i \in \cI_{H(\mathfrak D)} \sub \cI^*$,  contradicting $D_i\in\cD$. Thus $p$ violates \eqref{eq:k-thr} for $\cI=\cI_H$.
\end{proof}

\section{Proof of \texorpdfstring{\Cref{thm:spanning graphs}}{Theorem 1.11}}\label{sec:spanning graphs} By \cref{thm:k-thr}, it suffices to decompose each graph $H$ in the statement as $H=H_1 \cup H_2$ such that 
$p_c(H_i)=O(q(H))$ for $i=1,2$.
In the following subsections, we consider each of the graphs in the statement in turn. Recall from \cref{subsec:exp_thr} that for any $H$, we have $\pE(H) \le q(H)$.

\subsection{Bounded-degree spanning trees}\label{subsec:bdst}
Suppose $H$ is a spanning tree with maximum degree $d$, where $d$ is a constant independent of $n$.

We first observe that in this case, we have $q(H) \geq \pE(H) = \Omega(1/n)$. To see the lower bound on $\pE(H)$, let $X_H$ count copies of $H$ in $G_{n,p}$. Since $H$ has $n-1$ edges and at most $n!$ labeled copies in $K_n$, we have $\mathbb E_pX_H \le n!p^{n-1}$, which is less than $1/2$ for, say, $p=1/(2n)$. This violates the defining condition for $\pE(H)$ in \eqref{def:pE}, so we have $\pE(H)=\Omega(1/n)$.

Our proof proceeds by decomposing $H$ into two forests, each with at most $2n/3+1$ vertices, so that we can apply \cref{thm:embedding_trees}.
We use the following standard fact about trees.
\begin{fact}\label{fact:centroid}
Any $n$-vertex tree has a vertex $v$, called a \emph{centroid}, such that after removing $v$ from the tree, each connected component has at most $n/2$ vertices.
\end{fact}
Given the spanning tree $H$ with $n$ vertices, let $v$ be a centroid of $H$. Let $A_0, B_0\subseteq V(H) \setminus \{v\}$ be formed as follows. Sort the components of $H\setminus \{v\}$ in decreasing order of size. Put all the vertices from each component into $B_0$, stopping when adding the vertices of the next component would bring $|B_0|$ to at least $2n/3$. Put the rest of the components' vertices into $A_0$. We claim that
\[|A_0|, |B_0| < 2n/3.\]
The upper bound on $|B_0|$ is obvious from the construction. For the upper bound on $|A_0|$, since the components of $H \setminus \{v\}$ have at most $n/2$ vertices each, $B_0$ is nonempty; furthermore, the components are in decreasing order, so if $|B_0|< n/3$, then we can add the next component to $B_0$ still maintaining $|B_0| < 2n/3$. Thus, $|B_0| \geq n/3$, so $|A_0|=n-1-|B_0|<2n/3$.

Let $A = A_0\cup\{v\}$ and $B = B_0 \cup\{v\}$. 
Each of the graphs $T_A:=H[A]$ and $T_B:=H[B]$ (as usual, $G[X]$ denotes the subgraph of $G$ induced by the vertex subset $X$) is a tree on at most $2n/3+1$ 
vertices with maximum degree at most $d$. Indeed, $A$ consists of $v$ together with a collection of whole components of $H \setminus \{v\}$, so the induced graph $H[A]$ is connected and acyclic; similarly for $B$.   
\Cref{thm:embedding_trees} now implies that $p_c(T_A),p_c(T_B)=O(1/n)$. Also, we have $H = T_A \cup T_B$, so we have found the desired decomposition of $H$.  

\subsection{Powers of a Hamiltonian cycle}
Fix an integer $m \ge 1$, and let $H$ be the $m$-th power of a Hamiltonian cycle. Since $e_H=mn$ and there are at most $(n-1)!/2$ copies of $H$ in $K_n$, we have
\[\mathbb E_pX_H\leq (n-1)! p^{mn}/2,\]
yielding $q(H) \geq \pE(H) = \Omega(n^{-1/m})$. 

Suppose first that $m=1$, so that $H$ is a Hamiltonian cycle $C_n$. Decompose
$H$ into two paths $H_1$ and $H_2$ of lengths
$\lfloor n/2\rfloor$ and $\lceil n/2\rceil$, respectively. By \cref{thm:embedding_trees}, if $p = c/n$ for a sufficiently large constant $c$, then $G_{n,p}$ contains copies of both $H_1$ and $H_2$ with high probability. Therefore, $p_c(H_i) = O(q(H))$ for $i=1,2$.

For $m\geq 2$, recall that $p_k(H)$ is non-increasing in $k$, so $p_2(H) \leq p_1(H) = p_c(H)$.
For $m=2$,  Kahn, Narayanan, and the third author~\cite{kahn2021threshold} proved that $p_c(H)=O(n^{-1/m})$, while for $m\geq 3$ this follows from an earlier result of Riordan~\cite{riordan2000spanning}. Therefore, $p_2(H)=O(n^{-1/m})=O(q(H))$.

\subsection{\texorpdfstring{$F$}{F}-factor for fixed \texorpdfstring{$F$}{F}}

Finally, we consider the case where $H$ is an $F$-factor for a fixed graph $F$ with at least one edge. Note that $v_F$ and $e_F$ are constants independent of $n$, and we will often suppress dependency of constants on $F$ throughout the subsection. Define that the \textit{1-density} of $F$ to be
\[m_1(F) = \max\left\{\frac{e_J}{v_J-1}: \text{$J$ is a subgraph of $F$ with at least one edge}\right\}.\]

We first show that $q(H)\ge  \pE(H)=\Omega(n^{-1/m_1(F)})$. Fix a subgraph $J$ of $F$ that attains $m_1(F)$, and let $I$ be the subgraph of $H$ formed by taking the union of the $n/v_F$ vertex-disjoint copies of $J$, obtained by selecting one such copy from each copy of $F$ in $H$. Then 
\[\text{$v_I=v_J \cdot n/v_F$, \quad $e_I=e_J\cdot n/v_F$, \quad and \quad $\aut(I) \geq (n/v_F)!\cdot \aut(J)^{n/v_F}$,}\]
so
\[\mathbb E_pX_I=\frac{(n)_{v_I}}{\aut(I)}\cdot p^{e_I}\le \frac{n^{v_J\cdot n/v_F}}{(n/v_F)!\cdot \aut(J)^{n/v_F}}\cdot p^{e_J\cdot n/v_F} \le \left(\frac{e\cdot v_F}{\aut(J)}n^{v_J-1}p^{e_J}\right)^{n/v_F}.\]
After substituting $p=cn^{-(v_J-1)/e_J}$, the quantity inside the parentheses becomes $ev_F/\aut(J) \cdot c^{e_J}$, which is independent of $n$. Since $J$ is fixed, choose $c=c(F)$ sufficiently small that $ev_F/\aut(J)\cdot c^{e_J}<1/2$. Then $\mathbb E_pX_I<1/2$, yielding the desired conclusion.

Now, partition the copies of $F$ forming $H$ so that $H_1$ consists of $\ell:=\lceil n/(2v_F) \rceil$ vertex-disjoint copies of $F$, and let $H_2$ consist of the remaining copies. Thus $H = H_1\cup H_2$. Since $H_2$ is isomorphic to a subgraph of $H_1$, we have $p_c(H_2) \le p_c(H_1)$, so it suffices to show that $p_c(H_1)=O\bigl(n^{-1/m_1(F)}\bigr)$.

Let $p = C n^{-1/m_1(F)}$ where $C=C(F)$ is a sufficiently large constant. Fix a realization $G$ of $G_{n,p}$. Repeatedly choose a copy of $F$ in the current graph and delete its $v_F$ vertices. If this procedure terminates before $\ell$ copies have been chosen, then the remaining graph is $F$-free and has more than $n/2$ vertices. Consequently, $G$ contains an $F$-free set of $n/2$ vertices. It therefore suffices to show that, for $G \sim G_{n,p}$, with high probability, every set of $n/2$ vertices spans a copy of $F$. 
We prove this using the Janson inequalities.

Fix $S \sub [n]$ with $|S|=n/2$, and let $Y_S$ count the copies of $F$ in $G[S]$. Then
\[\mu:=\mathbb E[Y_S]=\Theta(n^{v_F}p^{e_F})=\Theta(C^{e_F}n^{v_F-e_F/m_1(F)}).\]
Since $m_1(F) \ge \frac{e_F}{v_F-1}$, there is a constant $a=a(F)$ such that
\[\mu \ge a\cdot C^{e_F} n=:C'n.\]
In particular, $C'$ can be made arbitrarily large by choosing $C$ sufficiently large.

Next, with $t$ the number of copies of $F$ in the complete graph $K_{n/2}$, label the copies from $1$ to $t$. For each $j = 1,\ldots, t$, let $X_j$ be the indicator random variable for the presence of the $j$-th copy of $F$ in $G[S]$. 
 For distinct $i,j \in [t]$, write $i\sim j$ if the corresponding copies of $F$ 
 share at least one edge. Let $\Delta = \sum_{(i,j):\,i\sim j} \EE[X_iX_j]$, where the sum is over ordered pairs. 

If $\Delta \leq \mu$, then 
\cref{thm:Janson} gives 
$\pr(Y_S = 0) \leq e^{-\mu + \Delta/2} \leq e^{-C'n/2}.$
Taking a union bound over the choices of the set $S$, the probability that there exists a set $S\subseteq V(G)$ of size $n/2$ such that $G[S]$ fails to contain a copy of $F$ is at most 
$$\binom{n}{n/2} e^{-C'n/2} \leq 2^{n} e^{-C'n/2} = o(1).$$

Next, suppose that 
$\Delta \geq \mu$. For each graph $J$ that can arise as the intersection of two distinct copies of $F$ sharing at least one edge, let $\gD_J$ denote the contribution to $\gD$ from ordered pairs of copies with intersection isomorphic to $J$. In particular, $J$ is a subgraph of $F$ with $e_J \ge 1$. The union of such a pair of copies of $F$ has $2v_F-v_J$ vertices and $2e_F-e_J$ edges, so  
$$\gD_J=O(n^{2v_F-v_J}p^{2e_F-e_J}) = O\left(\frac{\mu^2}{n^{v_J}p^{e_J}}\right).$$ Since $p = Cn^{-1/m_1(F)} \geq Cn^{-(v_J-1)/e_J}$, we have $n^{v_J}p^{e_J} \ge C^{e_J}n$. Consequently, 
$\gD_J=O(C^{-e_J}\mu^2/n)=O(C^{-1}\mu^2/n)$. Since there are only $O_F(1)$ possible intersection 
types, 
\[\Delta \le \eps\mu^2/n,\]
where $\eps=\eps(C) \to 0$ as $C \to \infty$.
Now, \cref{thm:extJanson} yields $\pr(Y_S = 0) \leq e^{-\mu^2/(2\Delta)} \leq e^{-C'' n},$ where $C''$ can be made large by taking $C$ sufficiently large. Again, taking a union bound over the choices of the set $S$, the probability that there exists a set $S\subseteq V(G)$ of size $n/2$ such that $G[S]$ fails to contain a copy of $F$ is at most $$\binom{n}{n/2} e^{-C''n} \leq 2^{n} e^{-C''n} = o(1),$$ 
concluding the proof.

\section{Forest decomposition and proof of \texorpdfstring{\Cref{thm:bounded degeneracy}}{Theorem 1.12}}\label{sec: bounded degeneracy}

The proof of \Cref{thm:bounded degeneracy} reduces bounded-degeneracy graphs to forests using the Nash-Williams theorem (\Cref{thm:NW}). The main ingredient is the following decomposition result for forests.

\begin{lem}\label{lem:all trees}
Every forest $F$ admits a $6$-decomposition $(\tilde F_1, \ldots, \tilde F_6)$ such that $p_c(\tilde F_i) = O(q(F))$ for each $i\in[6]$. 
\end{lem}

\begin{proof}[Proof of \cref{thm:bounded degeneracy}]
Let $\d(H)\le D$. By \Cref{cor:NW}, the arboricity of $H$ is at
most $D$, and hence $H$ can be written as the union of $m\leq D$ forests
$F_1,\ldots,F_m$. For each forest $F_j$, since $F_j\subseteq H$, we have
$\cI_H\subseteq \cI_{F_j}$, and therefore $q(F_j)\le q(H)$.

By \cref{lem:all trees}, each forest $F_j$ admits a $6$-decomposition $(\tilde F_{j, 1}, \ldots, \tilde F_{j, 6})$ with
$p_c(\tilde F_{j, i}) = O(q(F_j)) \le O(q(H))$ for each $i\in[6]$. Thus the sequence of all the $\tilde F_{j,i}$'s is a $6D$-decomposition of $H$ with each piece having threshold $O(q(H))$.
Applying \cref{thm:k-thr} yields the conclusion.
\end{proof}

\begin{proof}[Proof of \cref{lem:all trees}] The proof consists of four main steps. Recall our convention that $F = F_1\cup \cdots \cup F_m$ means $E(F) = E(F_1)\cup \cdots \cup E(F_m)$.

\nin \textbf{Step 1 (Decomposition):} We decompose $F$ into the union of three \emph{star forests} (forests where each component is a star) $F_1, F_2, F_3$, each of which spans at most $(1-\Omega(1))n$ vertices. Every $F_i$ is further partitioned into $L_i$ and $H_i$, which comprise its low-degree and high-degree star components, respectively. 

\noindent \textbf{Step 2 (Low-degree embedding):} We greedily embed each $L_i$ into $G_{n,p}$ and show that this embedding succeeds with high probability for $p = O(q(L_i)) \le O(q(F))$.

\noindent \textbf{Step 3 (Size bound on high-degree stars):} We bound the number of vertices in the high-degree stars, showing that $v_{H_i}\le\delta n$ for all $i$, where
$\delta>0$ can be made arbitrarily small by choosing the degree cutoff sufficiently large.

\noindent \textbf{Step 4 (High-degree embedding):} We use the above size bound to prove that $G_{n,p}$ contains $H_i$ with probability at least $1/2$ for $p = O(q(F))$.

\begin{claim}[Decomposition]\label{claim:star_forest_decomp}
Fix $\epsilon = 1/10$. For sufficiently large $n$, any forest $F$ on at most $n$ vertices can be written as a union $F = F_1\cup F_2\cup F_3$ where each $F_i$ is a star forest on at most $(1-\epsilon)n$ vertices.
\end{claim}
\begin{proof}

We may assume that $F$ is a tree on exactly $n$ vertices; indeed, first add isolated vertices so that $F$ has exactly $n$ vertices, and then add edges between its components to obtain a tree $T$. It suffices to find the desired decomposition of $T$: after deleting the added vertices and edges, each piece remains a star forest and its number of vertices cannot increase.

To begin, pick a leaf $v$ of $F$ as the root; let $V_i$ denote the set of vertices at distance $i$ from $v$ (so in particular $V_0 = \{v\}$). Since $F$ is a tree, each of its edges has one end in $V_i$ and the other in $V_{i+1}$ for some $i$. If $i$ is even, place the edge in $F_1$, and otherwise place the edge in $F_2$. Both $F_1$ and $F_2$ are star forests. For now, $F_1$ and $F_2$ might both span more than $(1-\epsilon)n$ vertices, so we next move some edges to $F_3$.

Suppose $F$ has $\ell$ leaves and $t$ vertices of degree $2$. Then, by double-counting the edges of $F$, we have $2(n-1) \geq \ell + 2t + 3(n-\ell-t)$, and hence $2\ell+t \geq n+2$. Thus either $F$ has at least $\ceil{\epsilon n}$ leaves, or it has at least $(1-2\epsilon)n$ vertices of degree $2$.

First suppose $F$ has at least $\ceil{\epsilon n}$ leaves. Move exactly $\ceil{\epsilon n}$ edges incident to leaves into $F_3$. The corresponding leaves are then isolated in both $F_1$ and $F_2$; after deleting them, each of these graphs has at most $(1-\eps)n$ vertices. Moreover, $F_3$ is then a star forest with at most $2\lceil \eps n \rceil \le (1-\eps)n$ vertices. We have $F = F_1 \cup F_2 \cup F_3$, so this concludes the proof for this case. 

Now, suppose instead that $F$ has at least $(1-2\epsilon)n$ vertices of degree $2$. For each maximal path $P=(u_1, \ldots, u_r)$ (where $u_1$ is the endpoint closer to the root $v$) of degree-$2$ vertices in $F$, select $u_2, u_5, u_8, \ldots$. Stop selecting vertices after we have selected $\lceil \epsilon n \rceil$ in total. 
We claim that at least $\lceil\epsilon n\rceil$ vertices are available for selection. Let $s$ be the number of maximal paths of degree-$2$ vertices.
A path containing $r$ such vertices contributes at least
$(r-1)/3$ selected vertices, so the total number available is at
least $(t-s)/3$. Suppressing all degree-$2$ vertices produces a
tree on $n-t$ vertices, and hence $s\le n-t-1\le 2\epsilon n.$ 
Therefore, the number available is at least $\frac{t-s}{3}\ge \frac{1-4\epsilon}{3}n>\lceil \epsilon n\rceil.$ 

Note that the selected vertices on each path are at distance at least 3. The same holds for selected vertices on distinct paths. 
Indeed, the path between two such vertices passes through a vertex of degree at least $3$, as otherwise the two degree-$2$ paths would belong to the same maximal path. It also passes through the first degree-$2$ vertex of the path farther from the root.

Put both edges incident to each selected vertex into $F_3$. Each selected vertex is then isolated in $F_1$ and $F_2$; after deleting them, each of these graphs has at most $(1-\eps)n$ vertices. 
Furthermore, the vertices spanned by $F_3$ are contained in the selected degree-2 vertices and their neighbors, meaning there are $3\ceil{\epsilon n} \leq (1-\epsilon)n$ vertices in $F_3$. Each $F_i$ is a star forest, and once again we have $F = F_1 \cup F_2 \cup F_3$, so the proof is complete.
\end{proof}

After decomposing $F = F_1\cup F_2 \cup F_3$ as in the above claim, by symmetry it suffices to consider 
$F_1$.  
If $F_1$ is edgeless, the desired conclusion is immediate, so assume that $F_1$ has at least one edge. Recall the definition of $\tilde \pE$ from \Cref{subsec:exp_thr}. Set $f(n)=n\tilde \pE(F_1)$. Fix a constant $c_1>0$, to be chosen sufficiently large later, and 
write $F_1 = L_1 \cup H_1$, where $L_1$ consists of the star components of $F_1$ with degree at most $c_1 f(n)$, 
and $H_1$ consists of the remaining components. 

\begin{claim}[Low-degree embedding]\label{claim:embedding-A1} There exists a constant $c=c(c_1, \eps)>0$ such that, for $p = \min\{c f(n)/n, 1\}$, the random graph $G_{n,p}$ contains a copy of $L_1$ with high probability.
\end{claim}
\begin{proof} We may assume that $f(n) \ge 1/c_1$, since otherwise 
$L_1$ is edgeless 
and there is nothing to prove. 
We may also assume that $p<1$, as otherwise the claim holds trivially. Since isolated vertices do not affect graph containment, we delete them from $L_1$.

Let $x$ be the number of star components of $L_1$, and let $y$ be their total number of leaves. 
For a component consisting of a single edge, 
we think of one vertex as a center and the other as a leaf. 
Then $x+y=v_{L_1} \leq (1-\epsilon)n$. Partition the vertex set $[n]$ 
into $X\cup Y$, where $|X|=x+\epsilon n/2$ and $|Y| = n-|X| \geq y+\epsilon n/2$. 

We embed the centers of the stars of $L_1$ into $X$ and their leaves into $Y$. Initially, set $Y' = Y$. Process the stars one at a time. For each star $S$, inspect 
the vertices of $X$ that have not previously been used or discarded, one at a time, until finding a vertex $w$ with at least 
$c_1 f(n)$ neighbors in the current set $Y'$. Any inspected vertex in $X$ with fewer than $c_1f(n)$ neighbors in $Y'$ is discarded and not used later in the procedure. Once a suitable vertex $w$ is found, embed the center of $S$ into $w$, embed its leaves into arbitrary neighbors of $w$ in $Y'$, and remove those leaves from $Y'$. 

We claim that this process successfully embeds all of $L_1$ with high probability.
Throughout the algorithm, at most $y$ vertices are removed from $Y$, and hence $|Y'|\ge|Y|-y\ge \eps n/2$.  When a new vertex $v\in X$ is inspected, the current set $Y'$ has been determined entirely by the preceding steps, which involve only previously inspected vertices of $X$. Also, no edge from $v$ to the current set $Y'$ has previously been exposed. Thus, conditional on the entire previous history, the  number of neighbors of $v$ in the current $Y'$ is distributed as $\bin(|Y'|,p)$ and therefore stochastically dominates $Z \sim \bin(\ceil{\epsilon n/2},p)$. The mean of $Z$ satisfies 
\[\mu:=\mathbb E[Z] \ge \frac{\eps n}{2}\cdot\frac{cf(n)}{n}
=\frac{c\eps}{2}f(n).\]
Choose $c$ sufficiently large that $c_1f(n) \le \frac{\mu}{2}$. Then 
the Chernoff bound gives
\[
\pr\bigl(Z<c_1f(n)\bigr)
\le \pr(Z<\mu/2)
\le e^{-\mu/8}
\le \exp\left(-\frac{c\epsilon}{16c_1}\right).
\]
By increasing $c$ further, we may ensure that the last quantity is
at most $\eps/10$.

It follows, by a standard coupling, that the total number of
discarded vertices is stochastically dominated by $\bin(n,\epsilon/10).$ 
Another application of the Chernoff bound gives that the probability of discarding at least $\epsilon n/2$ vertices is $o(1)$. Since $|X|=x+\eps n/2$, on the complementary event the procedure cannot exhaust $X$ before finding centers for all $x$ stars and embedding their neighbors into $Y$. Thus it embeds all of $L_1$ with high probability.
\end{proof}

Now, we turn to $H_1$, which by construction has no isolated vertices and consists of stars of degree at least $L:=\lceil c_1f(n)\rceil$. We first produce a lower bound on the graphic expectation threshold of $H_1$. 

\begin{mydef}
Let $\Delta$ be the maximum degree of $H_1$. For each integer $\ell$ with $L\leq \ell\leq \Delta$, let $M_\ell$ denote the number of components of $H_1$ that are stars of degree at least $\ell$. 
Define 
\[\tilde q(H_1) = \max_{L\leq \ell \leq \Delta} \frac{\ell \cdot M_\ell^{1/\ell}}{n^{1+1/\ell}}.\]
\end{mydef}

We claim that $\tilde \pE(H_1) \geq \Omega(\tilde q(H_1))$. To see this, 
for any $L \le \ell \le \Delta$, let  $I$ be the subgraph of $H_1$ 
consisting of $M_\ell$ vertex-disjoint stars of degree exactly $\ell$, obtained by deleting extra leaves from stars of larger degree. 
Since $v_I=(\ell+1)M_\ell$, $e_I=\ell M_\ell$, and $\aut(I) \ge (\ell!)^{M_\ell}M_\ell!$, using $r! \ge (r/e)^r$, 
$$\tilde \pE(H_1) \geq \left(\frac{\aut(I)}{n^{v_I}}\right)^{1/e_I} \ge \left(\frac{(\ell!)^{M_\ell} M_\ell!}{n^{(\ell+1)M_\ell}}\right)^{1/(\ell M_\ell)} \ge e^{-2}\cdot \frac{\ell M_\ell^{1/\ell}}{n^{1+1/\ell}}.
$$
Taking the maximum over $L \le \ell \le \gD$ yields $\tilde \pE(H_1) \ge e^{-2}\tilde q(H_1)$,  as desired. 

We thus have
\beq{eq:tilde-q}\tilde q(H_1) \leq e^2\tilde\pE(H_1) \leq e^2 \tilde \pE(F_1) =e^2f(n)/n.\enq We will use this inequality in the next two claims.

\begin{claim}[Size bound on high-degree stars] \label{claim:B1-small} The number of 
vertices of $H_1$ is at most $\delta n$, where $\delta>0$ can be made arbitrarily small by choosing $c_1$ sufficiently large.
\end{claim}

\begin{proof} Since $H_1$ is a star forest with no isolated vertices, the number of its vertices is at most $2e_{H_1}$ (because every star with $d~(\ge 1)$ edges has $d+1 \le 2d$ vertices), so it suffices to upper bound $e_{H_1}$. Note that, by 
rearranging \eqref{eq:tilde-q}, we have that for every $\ell$ with $L \le \ell \le \Delta$,
$$M_\ell  \leq n\left(e^2f(n)/\ell\right)^\ell.$$ 
Since $M_\ell$ is the number of components of $H_1$ which are stars of degree at least $\ell$, 
we have 
$$e_{H_1} \leq \sum_{\ell=L}^{\Delta} \ell M_\ell \leq n \sum_{\ell=L}^\infty \ell \left(\frac{e^2f(n)}{\ell}\right)^\ell \le n \sum_{\ell=L}^\infty \ell \left(\frac{e^2}{c_1}\right)^\ell \le \delta n/2,$$
where the last inequality holds for a sufficiently large $c_1$.
This yields the conclusion.
\end{proof}

Finally, we use the fact that $H_1$ has at most $\gd n$ vertices 
to embed it into $G_{n,p}$ with $p=cf(n)/n$.

\begin{claim}[High-degree embedding] \label{claim:embedding-B1} There exists a constant $c=c(c_1, \eps)$ such that, for $p=\min\{cf(n)/n,1\}$, the random graph $G_{n,p}$ contains a copy of $H_1$ with probability at least $1/2$.
\end{claim}
\begin{proof} 
Choose $c_1$ sufficiently large that the constant $\gd$ in \Cref{claim:B1-small} satisfies $\gd \le 1/48$.
We may assume $p<1$, as otherwise the claim is immediate.

We partition the vertex set $[n]$ 
into $X\cup Y$, each with $n/2$ vertices. We embed the centers of the stars of $H_1$ into $X$ and their leaves into $Y$. 
Initially, set $X'=X$ and $Y'=Y$, 
and process the stars one at a time (in an arbitrary order). To embed a star $S$ of degree $\ell$, inspect the vertices of $X'$ one at a time. If an inspected vertex has fewer than $\ell$ neighbors in $Y'$, remove it from $X'$ and continue.  
Upon finding a vertex $w$ with at least $\ell$ neighbors in $Y'$, 
embed the center of $S$ into $w$ and its leaves into arbitrary $\ell$ such neighbors; then remove $w$ from $X'$ and chosen neighbors from $Y'$. If no suitable vertex remains in $X'$, the procedure fails. 

By \Cref{claim:B1-small}, $H_1$ has at most 
$\gd n \le n/48$ vertices, so throughout the procedure we have $|Y'|\ge n/2-\gd n \ge n/3$. When a vertex $v \in X'$ is inspected, the current set $Y'$ has been determined entirely by the preceding steps, which involve only previously inspected vertices of $X$. Also, no edge from $v$ to the current set $Y'$ has previously been exposed. 
Hence, conditional on the preceding history, the number of neighbors of $v$ in $Y'$ is distributed as $\bin(|Y'|, p)$ and therefore stochastically dominates $\bin(n/3, p)$. Therefore, when embedding a star of degree $\ell$, each inspected 
vertex has conditional probability at least
\[r_\ell:=\pr(\bin(n/3,p)\ge \ell)
\]
of being usable as the image of its center. Consequently, conditional on the preceding history, the number $T_\ell$ of vertices inspected until such a center is found is stochastically dominated by a geometric random variable with success probability $r_\ell$, and hence 
\beq{eq:T_ell} \text{$T_\ell$ has conditional  expectation  at most $1/r_\ell$ regardless of the preceding history.}\enq 
We now estimate $r_\ell$. Recalling \eqref{eq:tilde-q}, since $p=cf(n)/n$, we have $p \geq c_2 \tilde q(H_1)$, where $c_2:=c/e^2$ can be made arbitrarily large by choosing $c$ sufficiently large. We show that there is a constant $c_3:=c_2/(3e^4)~(=c/(3e^6))$ which can also be made arbitrarily large, such that, for every integer $\ell$ with $L\le \ell\le \Delta$,
\beq{eq:bin-estimate} r_\ell \geq \min\{1/6, c_3^\ell \cdot M_\ell / n\}.
\enq

Let $Z \sim \bin(n/3, p)$ and write $\mu=\mathbb E[Z]=np/3$. First suppose that $1\leq \ell \leq np/6$.  By the Paley-Zygmund inequality,
\[\pr(Z>\mu/2) \geq \frac{1}{4} \cdot \frac{\mu^2}{\EE[Z^2]}.\]
Since $\EE[Z^2]\le \mu+\mu^2$ and $\mu \ge 2\ell \ge 2$, it follows that
\[\pr(Z\ge \ell)\ge \pr(Z>\mu/2)\ge \frac14\cdot \frac{\mu^2}{\mu^2+\mu}\ge \frac16.
\]

Now suppose that $np/6\leq \ell\leq \Delta$. 
In this case we have $p\leq1/2$ since $\ell \le \gD < v_{H_1} \le \gd n \le n/48$, which would contradict $\ell \geq np/6$ if $p>1/2$. 
Since $\ell < n/48<n/3$, 
\[\pr(Z\ge \ell)\ge \binom{n/3}{\ell}p^\ell(1-p)^{n/3-\ell}.
\]
Since $p \le 1/2$, we have $1-p\ge e^{-2p}$. Together with $\ell \ge np/6$, this gives
\[(1-p)^{n/3-\ell}\ge \exp(-2p(n/3-\ell))\ge \exp(-2np/3)\ge e^{-4\ell}.
\]
Also, using $\binom{n/3}{\ell}
\ge\left(\frac{n}{3\ell}\right)^\ell$ and that $p \ge c_2 \tilde q(H_1)$ we obtain
\[\binom{n/3}{\ell}p^\ell \ge 
\left(\frac{np}{3\ell}\right)^\ell
\ge
\left(\frac{n}{3\ell}\right)^\ell
c_2^\ell
\frac{\ell^\ell M_\ell}{n^{\ell+1}}
=
\left(\frac{c_2}{3}\right)^\ell
\frac{M_\ell}{n}.
\]
Combining these estimates gives
\[
\pr\bigl(Z\ge\ell\bigr)
\ge
\left(\frac{c_2}{3e^4}\right)^\ell
\frac{M_\ell}{n}
=
\frac{c_3^\ell M_\ell}{n}.
\]
This proves \eqref{eq:bin-estimate}.

We now use \eqref{eq:bin-estimate} to bound the total number of vertices inspected by the embedding procedure. Let $N_\ell$ be the number of components of $H_1$ that are stars of degree \emph{exactly} $\ell$, and let $T$ denote the total number of vertices inspected during the embedding procedure. Given \eqref{eq:T_ell}, by the tower property and linearity of expectation,
\[\EE [T]\le \sum_{\ell=L}^\Delta \frac{N_\ell}{r_\ell}.
\]
Split the sum according to whether $c_3^\ell M_\ell/n\ge 1/6$. For such values of $\ell$, \eqref{eq:bin-estimate} gives $r_\ell \ge 1/6$, so their total contribution is at most 
$6\sum_\ell N_\ell\le 6v_{H_1}\le 6\delta n$. For the remaining values of $\ell$,  \eqref{eq:bin-estimate} gives $r_\ell \ge c_3^\ell M_\ell/n$, and since $N_\ell \le M_\ell$, 
their total contribution is at most $\sum_{\ell=L}^\gD\frac{nN_\ell}{c_3^\ell M_\ell} \le \sum_{\ell\ge1} n c_3^{-\ell}\le \frac{n}{c_3-1}.$

Now, choose $c$ so that $c_3 \ge 9$. 
Recalling that $\gd \le 1/48$, we then have $\EE[T]\le 6\delta n+n/(c_3-1) \le n/8+n/8=n/4$. So, by Markov's inequality,
\[\pr(T>n/2)\le 1/2.\]
On the event $\{T\le n/2\}$, the greedy embedding never exhausts the vertices of $X$, and therefore embeds every nontrivial component of $H_1$. Thus $G_{n,p}$ contains $H_1$ with probability at least $1/2$.
\end{proof}

We now complete the proof of the lemma. By \cref{claim:star_forest_decomp}, we can write $F = F_1\cup F_2 \cup F_3$ where each $F_i$ is a star forest spanning at most $(1-\epsilon) n$ vertices. Further decomposing each $F_i$ into its low-degree star components and its high-degree star components gives a $6$-decomposition $(L_1, H_1, L_2, H_2, L_3, H_3)$ of $F$. Claims \ref{claim:embedding-A1} and \ref{claim:embedding-B1} show that $p_c(L_i), p_c(H_i)=O(\tilde \pE(F_i))$ for each $i$. Since $F_i$ is a subgraph of $F$, we have $\tilde \pE(F_i)=O(\pE(F_i))\le O(\pE(F)) \le  O(q(F))$. Thus every piece has threshold $O(q(F))$, completing the proof.
\end{proof}

\section{Proof of \texorpdfstring{\Cref{thm:general graph containment}}{Theorem 1.13}}\label{sec:general graph containment}

We first show that the assumed upper bound on the maximum degree implies an upper bound on the number of automorphisms of a graph.

\begin{lem}\label{lem: max deg aut bound}
Every 
graph $I$ with $v_I \leq n$, 
no isolated vertices, and maximum degree $\Delta \leq n^{1-\epsilon}$ satisfies 
$\aut(I) \leq n^{(1-\epsilon/2)v_I}$.
\end{lem}
\begin{proof}
We use a component-by-component count. If $B$ is a connected graph 
with $b:=v_B\ge2$, then an automorphism of $B$ is determined by the image of one starting vertex and then by the images chosen during a depth-first search of a spanning tree of $B$. Hence $\aut(B)\le b\Delta^{b-1}.$

Now write the connected components of $I$ as $m_j$ copies of pairwise non-isomorphic connected graphs $B_j$, where $b_j:=v_{B_j}$, so 
$\sum_j m_jb_j=v_I$. Then
\[\aut(I)=\prod_j m_j!\,\aut(B_j)^{m_j}.
\]
For each $j$, using $m_j!\le m_j^{m_j}$ and $m_jb_j\le n$, we get
\[
 m_j!\,\aut(B_j)^{m_j}
 \le m_j^{m_j}(b_j\Delta^{b_j-1})^{m_j}
 = (m_jb_j)^{m_j}\Delta^{m_j(b_j-1)}
 \le n^{m_j} n^{(1-\epsilon)m_j(b_j-1)}.
\]
The exponent on the right is
\[m_j+(1-\epsilon)m_j(b_j-1)=(1-\epsilon)m_jb_j+\epsilon m_j\le (1-\epsilon/2)m_jb_j,
\]
since $b_j\ge2$. Multiplying over all $j$ gives
\[\aut(I)\le \prod_j n^{(1-\epsilon/2)m_jb_j}=n^{(1-\epsilon/2)v_I},\] as desired.
\end{proof}

\begin{proof}[Proof of \cref{thm:general graph containment}]

Set $C' := \max(C, 1)$, 
$k' := \ceil{13C'/\epsilon}$, and $D := 5k'$. Let $H$ be a graph on at most $n$ vertices with $\d(H) \leq C\frac{\log n}{\log\log n}$ and maximum degree at most $n^{1-\epsilon}$. 

If $\d(H) \leq D$, then applying \cref{thm:bounded degeneracy}, there exists a constant $k'' = k''(D)$ such that $p_{k''}(H) = O(q(H))$. Thus, we may assume that $D \leq \d(H) \leq C\frac{\log n}{\log\log n}$. By \cref{thm:k-thr}, it is enough to find a $k'$-decomposition $(H_1, \ldots, H_{k'})$ of $H$ with $p_c(H_i) = O(q(H))$ for all $i$; this would prove \cref{thm:general graph containment} with $k = \max\{k', k''\}$.

Recall that $\pE(H)=\Theta(\tilde \pE(H))$ and $\pE(H) \leq q(H)$. 
It suffices to find a $k'$-decomposition $(H_1, \ldots, H_{k'})$ of $H$
such that
\[
\tilde \pE(H_i)
\le (\log n)^{-3.1}\tilde \pE(H)
\qquad\text{for every }i\in[k'].
\]
Indeed, by \Cref{thm:DKP}, this would imply
\[
p_c(H_i)
=O\bigl(\pE(H_i)\log^3n\bigr)=O\bigl(\tilde \pE(H_i)\log^3n\bigr)\le O\bigl((\log n)^{-0.1}\tilde \pE(H)\bigr)=o(q(H)).
\]
Equivalently, it suffices to prove that
\begin{align}\label{eq:graphic exp goal} \log \tilde \pE(H_i) \leq -3.1\log\log n + \log \tilde \pE(H).\end{align}

Let $(H_1, \ldots, H_{k'})$ be a $k'$-decomposition of $H$ 
such that $\d(H_i) \leq \ceil{\d(H) / k'}$ for every $i \in [k]$, which exists by \cref{claim: degen decomp}. By symmetry, it suffices to consider $H_1$. If $H_1$ is edgeless, then $p_c(H_1)=0$, so $p_c(H_1)=O(q(H))$ holds trivially. We may therefore assume that $H_1$ has at least one edge. 
Let $I$ be a maximizer in the definition of $\tilde \pE(H_1)$. We may assume $I$ has no isolated vertices since removing them only increases the quantity $(\aut(I)/n^{v_I})^{1/e_I}$. Using \Cref{lem: max deg aut bound}, 
\begin{align*}
\log \tilde \pE(H_1) &= \log\left(\left(\frac{\aut(I)}{n^{v_I}}\right)^{1/e_I}\right) \leq \frac{1}{e_I}((1-\epsilon/2)v_I\log n - v_I \log n) = -\frac{v_I}{e_I}\frac{\epsilon}{2}\log n.
\end{align*} By \cref{claim:degen vs mad}, we have $e_I/v_I \leq d(H_1) \leq \d(H_1) \leq \ceil{\d(H)/k'}$, so 
\[\log \tilde \pE(H_1) \leq -\frac{1}{\ceil{\d(H)/k'}} \frac{\epsilon}{2} \log n.\]

Since $\d(H)\ge D=5k'$, we have $1 \le \d(H)/(5k')$, and hence
\[
\left\lceil\frac{\d(H)}{k'}\right\rceil
\le \frac{\d(H)}{k'}+1
\le \frac{6\cdot \d(H)}{5k'},
\]
from which it follows that $\log\tilde \pE(H_i)
\le
-\frac{5k'\epsilon}{12\cdot \mathrm{degen}(H)}\log n$.
Since $k'=\lceil 13C'/\eps \rceil$, we have $k'\epsilon\ge13C'$, and therefore
\[
\log\tilde \pE(H_i)
\le
-\frac{65C'}{12\cdot \d(H)}\log n.
\]

On the other hand, let $J$ be a densest subgraph of $H$, satisfying $d(H) = e_J/v_J$. Then
\[\tilde \pE(H) \ge \left(\frac{\aut(J)}{n^{v_J}}\right)^{1/e_J} \ge n^{-v_J/e_J}.\]
By \Cref{claim:degen vs mad}, $e_J/v_J=d(H) \ge \d(H)/2$, so we have $\log \tilde\pE(H) \geq -\frac{2}{\mathrm{degen}(H)}\log n$. 
Hence, \eqref{eq:graphic exp goal} is satisfied if
$$-\frac{65 C'}{12\cdot \d(H)}\log n \leq -3.1\log\log n - \frac{2}{\d(H)}\log n;$$ equivalently,
$$\d(H)\cdot \frac{\log \log n}{\log n} \leq \frac{65C'/12 - 2}{3.1}.$$

Since $C'\geq 1$, we have $(65C'/12-2)/3.1 \ge (65/12-2)C'/3.1>C'\ge C$. 
The hypothesis of the theorem gives $\d(H) \cdot \log\log n/\log n \leq C$, so the desired inequality holds. 
\end{proof}

\section{Discussion and open problems}\label{sec:discussion}

The $k$-threshold reformulation established in \Cref{prop:equiv} suggests a new route toward \Cref{conj:DCC}. For graph containment properties, \Cref{thm:k-thr} makes this route concrete by reducing the problem of bounding $k$-thresholds to finding decompositions whose pieces have small ordinary thresholds. This motivates the following decomposition conjecture.

\begin{conj}\label{conj:DCC_k_str}
    There exist an integer $k \ge 2$ and a constant $L \ge 1$ such that, for every $n$ and every graph $H$ on at most $n$ vertices, $H$ admits a $k$-decomposition $(H_1, \ldots, H_k)$ satisfying $p_c(H_i) \le Lq(H)$ for every $i \in [k]$.
\end{conj}

By \Cref{thm:k-thr}, \Cref{conj:DCC_k_str} would imply that \Cref{conj:DCC_k} holds for every graph containment property. Our results establish this conclusion for several broad classes of target graphs, but the present methods do not appear to extend to arbitrary graphs.

One limitation arises from the logarithmic loss in the general bound $p_c(J)=O(\pE(J)\log^3 n)$. In the proof of \Cref{thm:general graph containment}, we decompose $H$ into graphs of lower degeneracy and show that their expectation thresholds are sufficiently smaller than that of $H$ to absorb this loss. Once $\d(H)=\omega(\log n/\log\log n)$, the argument no longer provides a sufficient gap to absorb the factor $\log^3 n$. Extending our approach beyond this range may therefore require sharper upper bounds on ordinary thresholds or a different decomposition argument.

The maximum-degree assumption in \Cref{thm:general graph containment} enters for a different reason: it is used in \Cref{lem: max deg aut bound} to bound the number of automorphisms of subgraphs of $H$. It would be interesting to determine whether this assumption can be removed.

One may isolate the difficulty created by high-degree vertices as follows. Given a graph $H$ of degeneracy $O(\log n/\log\log n)$, partition its vertices according to whether their degrees are at least $n^{1-\eps}$. Let $H_h$ and $H_l$ be the subgraphs induced by the high and low degree vertices, respectively, and let $H'$ consist of the edges between $H_h$ and $H_l$. The methods developed here can treat $H_h$ and $H_l$, leaving the bipartite graph $H'$ as the main remaining obstacle.

A separate question is whether more than two pieces are ever necessary for graph containment properties. Our results show that $k=2$ suffices for several classical spanning target graphs, but we do not know whether this holds for every graph.

\begin{question}\label{question:2-thr}
Does there exist a sequence of graphs $H=H_n$ such that $p_2(H) \gg q(H)$?
\end{question}

We note that even a proof of \Cref{conj:DCC_k_str} would not immediately settle \Cref{conj:DCC_k} for arbitrary increasing graph properties. Indeed, if $\cI$ is an increasing graph property and $G \in \cI$, then \Cref{thm:k-thr} gives $p_k(\cI)\le r_k(G)$. The conjectured decomposition bound would control $r_k(G)$ in terms of $q(G)=q(\cI_G)$, whereas the desired conclusion involves $q(\cI)$. Since $\cI_G \sub \cI$, one only has $q(\cI) \le q(G)$, which goes in the opposite direction of the comparison needed for our purposes. Extending the present approach to general increasing graph properties therefore appears to require additional ideas.

Finally, we remark that the proof of \Cref{thm:k-thr} extends with minor changes to $r$-uniform hypergraph properties for every fixed $r \ge 2$. It is natural to ask whether the resulting reduction yields any nontrivial bounds on $k$-thresholds for hypergraph containment properties.

\section*{Acknowledgments}
We are grateful to Noga Alon for suggesting the use of graph degeneracy as a key quantity to study for our decomposition arguments. 
Part of this work was carried out while the third and fourth authors visited the KIAS-KAIST Workshop on Current Challenges in Mathematics, and the authors gratefully acknowledge KIAS and KAIST for their support and hospitality. 

The authors used generative AI to assist with editing this manuscript. All original mathematical ideas and final language are our own.

JP was supported by NSF Grant DMS-2324978, NSF CAREER Grant DMS-2443706 and a Sloan Fellowship. 

\bibliographystyle{plain}
\bibliography{bibliography}

\end{document}